\documentclass[reqno]{amsart}

\newtheorem{thmnum}{}

\ifx\MyBeamer\undefined
\usepackage[breaklinks,colorlinks,pagebackref]{hyperref}
\usepackage[a4paper,margin=1in]{geometry}
\else
\if\MyBeamer t
\usepackage[breaklinks,colorlinks,backref]{hyperref}
\usepackage[a4paper,margin=1in]{geometry}
\usepackage{beamerarticle}
\fi\fi

\usepackage{amsmath}
\usepackage{amssymb}
\usepackage{mathrsfs}
\usepackage{amsthm}
\usepackage{aliascnt}
\usepackage[utf8]{inputenc}
\usepackage[T1]{fontenc}
\usepackage[nofancy]{svninfo}
\ifx\FrenchText\undefined
\usepackage[british]{babel}
\else
\usepackage[british,french]{babel}
\AtBeginDocument{%
\catcode`\:=12%
\catcode`\!=12%
}
\fi

\makeatletter

\newcounter{quotethmcnt}

\ifx\MyBeamer\undefined

\def\equationautorefname~#1\null{(#1)}
\def\itemautorefname~#1\null{#1}
\fi

\ifx\MyBeamer\undefined

\ifx\thmnum\undefined
\newtheorem{thmnum}{}[section]
\fi

\newcommand{\mynewthm}[3][]{%
  \newaliascnt{#2}{thmnum}%
  \newtheorem{#2}[#2]{#3}%
  \aliascntresetthe{#2}%
  \newtheorem*{#2*}{#3}%
  \expandafter\newcommand\csname #2autorefname\endcsname{#3}%
  \expandafter\renewcommand\csname the#2\endcsname{\thethmnum}%
}

\newtheorem*{clm}{Claim}
\newenvironment{clmprf}{%
  \begin{proof}[Proof of claim]%
  }{\end{proof}}

\else

\let\xxx=\frametitle
\def\frametitle#1{%
  \xxx{%
    \setbeamercolor*{math text}{use={titlelike,my math text},fg=titlelike.fg!80!my math text.fg}%
    #1}%
  \setbeamercolor{math text}{use=my math text,fg=my math text.fg}%
}

\newcommand{\beamerenv}[3]{%
\newenvironment<>{#1}%
{%
  \setbeamercolor{temp}{fg=structure.fg}%
  \setbeamercolor{structure}{fg=#2}%
  \setbeamercolor{block body}{use=structure,bg=structure.fg!5!white}%
  \begin{#3}%
}%
{\end{#3}\setbeamercolor{structure}{fg=temp.fg}}}

\newcommand{\mynewthm}[3][green!50!black]{%
  \newtheorem*{#2x}{#3}%
  \beamerenv{#2}{#1}{#2x}%
}

\beamerenv{other}{violet}{block}

\fi

\ifx\FrenchText\undefined
\newcommand{\myiffrench}[2]{#2}
\else
\newcommand{\myiffrench}[2]{\iflanguage{french}{#1}{#2}}
\fi

\theoremstyle{plain}
\mynewthm[red]{thm}{\myiffrench{Théorème}{Theorem}}
\mynewthm{prp}{Proposition}
\mynewthm[purple]{lem}{\myiffrench{Lemme}{Lemma}}
\mynewthm[orange]{fct}{\myiffrench{Fait}{Fact}}
\mynewthm[purple!50!white]{cor}{\myiffrench{Corollaire}{Corollary}}

\theoremstyle{definition}
\mynewthm[green!80!black]{dfn}{\myiffrench{Définition}{Definition}}
\mynewthm{hyp}{\myiffrench{Hypothèse}{Hypothesis}}
\mynewthm{conv}{Convention}
\mynewthm{conj}{Conjecture}
\mynewthm{ntn}{Notation}
\mynewthm{cst}{Construction}

\theoremstyle{remark}
\mynewthm[yellow!90!black]{rmk}{\myiffrench{Remarque}{Remark}}
\mynewthm{qst}{Question}
\mynewthm{exc}{\myiffrench{Exercice}{Exercise}}
\mynewthm{exm}{\myiffrench{Exemple}{Example}}

\ifx\MyBeamer\undefined

\newcommand{\myenumlabel}[1]{\textnormal{(\roman{#1})}}

\fi

\newcounter{cycprfcnt}
\newcounter{cycprffirst}
\newcommand{\cycprfpreamble}[1]%
{%
  \setcounter{cycprfcnt}{1}
  \setcounter{cycprffirst}{#1}
  \setlength{\itemindent}{0.5\leftmargin}%
  \setlength{\leftmargin}{0pt}%
  \newcommand{\cpcurr}{\myenumlabel{cycprfcnt}}%
  \newcommand{\cpnext}{\addtocounter{cycprfcnt}{1}\cpcurr}%
  \newcommand{\cpnum}[1]{\setcounter{cycprfcnt}{##1}\cpcurr}%
  \newcommand{\cpfirst}{\cpnum{1}}%
  \newcommand{\impnext}{\cpcurr{} $\Longrightarrow$ \cpnext.}%
  \newcommand{\impfirst}{\cpcurr{} $\Longrightarrow$ \cpfirst.}%
  \newcommand{\eqnext}{\cpcurr{} $\Longleftrightarrow$ \cpnext.}%
  \newcommand{\impnum}[2]{\cpnum{##1}{} $\Longrightarrow$ \cpnum{##2}.}%
  \def\makelabel##1{\ifnum\value{cycprffirst}=0\hspace{-0.7\itemindent}\setcounter{cycprffirst}{1}\fi##1}%
}%

{\begin{list}{\impnext}{\cycprfpreamble{#1}}}%
{\qedhere\end{list}}%

\newenvironment{cycprf*}[1][0]%
{\begin{list}{\impnext}{\cycprfpreamble{#1}}}%
{\end{list}}%

\def\indsym#1#2{%
  \setbox0=\hbox{$\m@th#1x$}%
  \kern\wd0%
  \hbox to 0pt{\hss$\m@th#1\mid$\hbox to 0pt{$\m@th#1^{#2}$\hss}\hss}%
  \lower.9\ht0\hbox to 0pt{\hss$\m@th#1\smile$\hss}%
  \kern\wd0}

\def\nindsym#1#2{%
  \setbox0=\hbox{$\m@th#1x$}%
  \kern\wd0%
  \hbox to 0pt{\hss$\m@th#1\not$\kern1.4\wd0\hss}
  \hbox to 0pt{\hss$\m@th#1\mid$\hbox to 0pt{$\m@th#1^{#2}$\hss}\hss}%
  \lower.9\ht0\hbox to 0pt{\hss$\m@th#1\smile$\hss}%
  \kern\wd0}

\def\dotminussym#1#2{%
  \setbox0=\hbox{$\m@th#1-$}%
  \kern.5\wd0%
  \hbox to 0pt{\hss\hbox{$\m@th#1-$}\hss}%
  \raise.6\ht0\hbox to 0pt{\hss$\m@th#1.$\hss}%
  \kern.5\wd0}

\newcommand{\rest}{{\restriction}}

\DeclareMathOperator{\tp}{tp}

\DeclareMathOperator{\tS}{S}

\newcommand{\bM}{\mathbf{M}}

\makeatother

\begin{document}

\title{Model theoretic stability and definability of types, after A.~\textsc{Grothendieck}}

\author{Itaï \textsc{Ben Yaacov}}

\address{Itaï \textsc{Ben Yaacov} \\
  Université Claude Bernard -- Lyon 1 \\
  Institut Camille Jordan, CNRS UMR 5208 \\
  43 boulevard du 11 novembre 1918 \\
  69622 Villeurbanne Cedex \\
  France}

\urladdr{\url{http://math.univ-lyon1.fr/~begnac/}}

\thanks{Research supported by the Institut Universitaire de France and ANR contract GruPoLoCo (ANR-11-JS01-008).}
\thanks{The author wishes to thank A.~\textsc{Berenstein}, S.~\textsc{Ferri} and T.~\textsc{Tsankov} for various discussions and ideas.}

\svnInfo $Id: Grothendieck.tex 1918 2014-09-16 10:59:31Z begnac $
\thanks{\textit{Revision} {\svnInfoRevision} \textit{of} \today}

\keywords{stable formula ; order property ; definable type ; relative weak compactness}
\subjclass[2010]{03C45 ; 46E15}

\begin{abstract}
  We point out how the ``Fundamental Theorem of Stability Theory'', namely the equivalence between the ``non order property'' and definability of types, proved by Shelah in the 1970s, is in fact an immediate consequence of Grothendieck's ``Critères de compacité'' from 1952.
  The familiar forms for the defining formulae then follow using Mazur's Lemma regarding weak convergence in Banach spaces.
\end{abstract}

\maketitle

In a meeting in Kolkata in January 2013, the author asked the audience who it was, and when, to have first defined the notion of a stable formula, and to the expected answer replied that, no, it had been Grothendieck, in the fifties.
This was meant as a joke, of course -- a more exact statement would be that in Théorème 6 and Proposition 7 of Grothendieck \cite{Grothendieck:CriteresDeCompacite} there appears a condition (see \autoref{eq:GrothendieckGroup} and \autoref{eq:GrothendieckSpace} below) which can be recognised as the ``non order property'' (NOP)\footnote{Iovino \cite{Iovino:StableBanach} points out that NOP also appears in a characterisation of reflexive Banach spaces due to James \cite{James:UniformlyNonSquareBanachSpaces}.
  For a direct connection between weak almost periodicity and reflexive Banach spaces see for example Megrelishvili \cite[Theorem~4.6]{Megrelishvili:FragmentabilityAndRepresentations}.}.
It took (us) a while longer to realise that one could ask, quite seriously, who first proved the ``Fundamental Theorem of Stability Theory'', namely, the equivalence between NOP and definability of types, and the answer would essentially be the same.
(As a model theoretic result, this was first proved by Shelah \cite{Shelah:ClassificationTheory}, probably in the seventies, generalising Morley's result that in a totally transcendental theory all types are definable.)

In everything that follows, if $X$ is a topological space then $C_b(X)$ denotes the Banach space of bounded, complex-valued functions on $X$, equipped with the supremum norm.
A subset $A \subseteq C_b(X)$ is \emph{relatively weakly compact} if it has compact closure in the weak topology on $C_b(X)$.

\begin{fct}[Grothendieck {\cite[Proposition~7]{Grothendieck:CriteresDeCompacite}}]
  \label{fct:GrothendieckGroup}
  Let $G$ be a topological group (in fact, it suffices that the product be separately continuous).
  Then the following are equivalent for a function $f \in C_b(G)$:
  \begin{enumerate}
  \item The function $f$ is \emph{weakly almost periodic}, i.e., the orbit $G \cdot f \subseteq C_b(G)$, say under right translation, is relatively weakly compact.
  \item Whenever $g_n,h_n \in G$ form two sequences we have
    \begin{gather}
      \label{eq:GrothendieckGroup}
      \lim_n \lim_m f(g_n h_m) = \lim_m \lim_n f(g_n h_m),
    \end{gather}
    as soon as both limits exist.
  \end{enumerate}
\end{fct}

This has been first brought to the author's attention by A.~\textsc{Berenstein} (see \cite{BenYaacov-Berenstein-Ferri:ReflexiveRepresentability}).
The first reference to \autoref{eq:GrothendieckGroup} as ``stability'' is probably the Krivine-Maurey stability \cite{Krivine-Maurey:EspacesDeBanachStables}, where $G$ is the additive group of a Banach space and $f(x) = \|x\|$ (or rather, $f(x) = \min \bigl( \|x\|, M \bigr)$ for some large $M$, since $f$ should be bounded -- in any case, Krivine and Maurey make no reference to Grothendieck's result).
As it happens, \autoref{fct:GrothendieckGroup} is a mere corollary of the following:

\begin{fct}[Grothendieck {\cite[Théorème~6]{Grothendieck:CriteresDeCompacite}}]
  \label{fct:GrothendieckSpace}
  Let $X$ be an arbitrary topological space, $X_0 \subseteq X$ a dense subset.
  Then the following are equivalent for a subset $A \subseteq C_b(X)$:
  \begin{enumerate}
  \item The set $A$ is relatively weakly compact in $C_b(X)$.
  \item The set $A$ is bounded, and whenever $f_n \in A$ and $x_n \in X_0$ form two sequences we have
    \begin{gather}
      \label{eq:GrothendieckSpace}
      \lim_n \lim_m f_n(x_m) = \lim_m \lim_n f_n(x_m),
    \end{gather}
    as soon as both limits exist.
  \end{enumerate}
\end{fct}

Our aim in this note is to point out how, modulo standard translations between syntactic and topological formulations, the Fundamental Theorem is an immediate corollary of \autoref{fct:GrothendieckSpace}.
In fact, we prove a version of the Fundamental Theorem relative to a single model, as in Krivine-Maurey stability, which in turn implies the usual version.
Our argument adapts a similar argument given in \cite{BenYaacov-Tsankov:WAP} in the context of $\aleph_0$-categorical structures.

Let us first recall a few definitions and facts regarding local types in standard first order logic.
We fix a formula $\varphi(x,y)$, where $x$ and $y$ are disjoint tuples of variables, say singletons, for simplicity.
If $\bM$ is a structure and $a \in \bM' \succeq \bM$, we define the \emph{$\varphi$-type} $\tp_\varphi(a/M)$ as the collection of all instances $\varphi(x,b)$, $b \in M$, such that $\varphi(a,b)$ holds, and let $\tS_\varphi(M)$ denote the space of all $\varphi$-types (we shall only consider $\varphi$-types over models).
We equip $\tS_\varphi(M)$ with the obvious topology, rendering it a compact, totally disconnected space.
The clopen subsets of $\tS_\varphi(M)$ are exactly those defined by Boolean combinations of instances of $\varphi$ over $M$ -- we call such a Boolean combination a \emph{$\varphi$-formula over $M$}.
We say that a formula $\psi(y)$ over $M$ \emph{defines} a $\varphi$-type $p \in \tS_\varphi(M)$ if for every $b \in M$ we have $\varphi(x,b) \in p(x)$ if and only if $\vDash \psi(b)$.

In the setting of continuous logic (see \cite{BenYaacov-Usvyatsov:CFO} or \cite{BenYaacov-Berenstein-Henson-Usvyatsov:NewtonMS}), the situation is essentially identical, \textit{mutatis mutandis}.
In fact, identifying \textit{True} with the value zero and \textit{False} with one, we can, and will, view the classical case described above as a special case of the following.
Recall first that a \emph{definable predicate} over $M$ is a continuous combination of (possibly infinitely many) formulae over $M$ (formulae being, by definition, finite syntactic objects), or equivalently, a uniform limit of formulae over $M$, or yet equivalently, a continuous function on $\tS_n(M)$ where $n$ is the number of arguments.
For all semantic intents and purposes definable predicates are indistinguishable from formulae, and every $\{0,1\}$-valued definable predicate is in fact a formula.
We define $p = \tp_\varphi(a/M)$ as the function which associates to each instance $\varphi(x,b)$, $b \in M$, the value $\varphi(a,b)$, which will then be denoted by $\varphi(p,b)$, or $\varphi^b(p)$.
We equip $\tS_\varphi(M)$ with the least topology in which all functions $\varphi^b$ (for $b \in M$) are continuous.
It is compact and Hausdorff, and every continuous function on $\tS_\varphi(M)$ can be expressed as a continuous combination of (possibly infinitely many, but at most countably many) functions of the form $\varphi^b$, or equivalently, as a uniform limit of finite continuous combinations -- such a definable predicate will be called a \emph{$\varphi$-predicate over $M$}.
A definable predicate $\psi(y)$ over $M$ defines $p \in \tS_\varphi(M)$ if $\varphi_p(b) = \psi(b)$ for all $b \in M$.

Finally, as per \cite[Appendix~B]{BenYaacov-Usvyatsov:CFO}, we say that $\varphi(x,y)$ is \emph{stable in a structure $\bM$} if whenever $a_n,b_n \in M$ form two sequences we have
\begin{gather}
  \label{eq:Stable}
  \lim_n \lim_m \varphi(a_n,b_m) = \lim_m \lim_n \varphi(a_n,b_m),
\end{gather}
as soon as both limits exist.
We say that $\varphi$ is \emph{stable in a theory $T$} if it is stable in every model of $T$.
We leave it to the reader to check that this is merely a rephrasing of the familiar NOP.

We first prove the Fundamental Theorem for stability inside a model.

\begin{thm}
  \label{thm:FundamentalTheoremModelWeak}
  Let $\varphi(x,y)$ be a formula stable in a structure $\bM$.
  Then every $p \in \tS_\varphi(M)$ is definable by a (unique) $\tilde \varphi$-predicate $\psi(y)$ over $M$, where $\tilde \varphi(y,x) = \varphi(x,y)$ (in the case of classical logic, a $\tilde \varphi$-formula).
\end{thm}
\begin{proof}
  Let $X = \tS_{\tilde \varphi}(M)$ and let $X_0 \subseteq X$ be the collection of those types realised in $M$, which is dense in $X$.
  Since $X$ is compact we have $C_b(X) = C(X)$.
  For $a \in M$ let $\varphi_a = \tilde \varphi^a$, so $A = \{\varphi_a : a \in M\} \subseteq C(X)$ is bounded (since every formula is).
  Thus, by \autoref{fct:GrothendieckSpace}, $\varphi$ is stable in $\bM$ if and only if $A$ is relatively weakly compact.

  Assume now that $\varphi$ is indeed stable in $\bM$, let $p \in \tS_\varphi(M)$, and let $a_i \in M$ form a net such that $\tp_\varphi(a_i/M) \rightarrow p$.
  By \autoref{fct:GrothendieckSpace} we may assume that $\varphi_{a_i}$ converges weakly to some $\psi \in C(X)$.
  Then $\psi$ is a $\tilde \varphi$-predicate over $M$, and for $b \in M$ we have
  \begin{gather*}
    \varphi(p,b) = \lim \varphi(a_i,b) = \psi(b),
  \end{gather*}
  as desired.
  The uniqueness of $\psi$ is by density of the realised types.
\end{proof}

With small variations, this appears in Pillay \cite[Corollary~2.3]{Pillay:IntroductionToStability} for classical logic, or in \cite[Theorem~B.4]{BenYaacov-Usvyatsov:CFO} for continuous logic (the latter also asserts that the defining formula is an \emph{increasing} continuous combination of instances of $\tilde \varphi$, which follows from \autoref{thm:FundamentalTheoremModelWeak} modulo Mazur's Lemma, see \autoref{cor:DefiningFormula}).
The Fundamental Theorem follows:

\begin{cor}[Fundamental Theorem of Stability]
  \label{cor:FundamentalTheoremTheory}
  Let $\varphi(x,y)$ be a formula and $T$ a theory.
  Then the following are equivalent.
  \begin{enumerate}
  \item The formula $\varphi$ is stable in $T$.
  \item For every model $\bM \vDash T$, every $\varphi$-type over $M$ is definable by a $\tilde \varphi$-predicate over $M$.
  \item For every model $\bM \vDash T$, every $\varphi$-type over $M$ is definable over $M$ (by some definable predicate).
  \item Let $\bM \vDash T$, and let $d$ denote the metric of uniform convergence on $\tS_\varphi(M)$ (i.e., $d(p,q) = \|\varphi(p,\cdot)-\varphi(q,\cdot)\|_\infty$).
    Then the density character of $\bigl( \tS_\varphi(M), d \bigr)$ is at most the density character of $M$ plus $|T|$ (in the classical settings all the distances are discrete and the density character is the same as the cardinal).
  \item There exists a cardinal $\kappa \geq |T|$ (in fact, any $\kappa = (\kappa_0 + |T|)^{\aleph_0}$ will do, and in the classical setting, any $\kappa \geq |T|$ will do) such that if $\bM \vDash T$, $|M| \leq \kappa$ then $|\tS_\varphi(M)| \leq \kappa$ as well.
  \end{enumerate}
\end{cor}
\begin{proof}
  The chain of implications (ii) $\Longrightarrow$ (iii) $\Longrightarrow \ldots \Longrightarrow$ (i) is straightforward and only requires elementary model theoretic methods and counting arguments.
  The most ``involved'' implication is (i) $\Longrightarrow$ (ii), which is by \autoref{thm:FundamentalTheoremModelWeak}.
\end{proof}

The Banach space formalism also allows us to obtain slight improvements quite easily.
First, regarding the case of stability in a single structure, we can improve \autoref{thm:FundamentalTheoremModelWeak} as follows.

\begin{thm}
  \label{thm:FundamentalTheoremModelStrong}
  Let $\varphi(x,y)$ be a formula and $\bM$ a structure, and for $a \in M$ let $\varphi_a = \tilde \varphi^a \colon q \mapsto \varphi(a,q)$.
  Then the following are equivalent.
  \begin{enumerate}
  \item \label{item:FundamentalTheoremModelStable}
    The formula $\varphi$ is stable in $\bM$.
  \item \label{item:FundamentalTheoremModelEmbedding}
    Every $p \in \tS_\varphi(M)$ is definable by a $\tilde \varphi$-predicate $\psi_p$ over $M$, and the map $p \mapsto \psi_p$ is a homeomorphic embedding of $\tS_\varphi(M)$ in the weak topology on $C\bigl( \tS_{\tilde \varphi}(M) \bigr)$.
  \item \label{item:FundamentalTheoremModelDefinability}
    If $p \in \tS_\varphi(M)$ is an accumulation point of a sequence $\bigl\{ \tp_\varphi(a_n/M) \bigr\}$ then there exists a sub-sequence $a_{n_k}$ such that $\varphi_{a_{n_k}}$ converges point-wise on $\tS_{\tilde \varphi}(M)$ to a definition of $p$.
  \end{enumerate}
  Moreover, in this case every $p \in \tS_\varphi(M)$ is the limit of a sequence of realised types.
\end{thm}
\begin{proof}
  We continue with the notations of the proof of \autoref{thm:FundamentalTheoremModelWeak}.
  \begin{cycprf*}[1]
  \item[\eqnext]
    If such a homeomorphism exists then $A$ is relatively weakly compact, and $\varphi$ is stable.
    For the converse, by the proof of \autoref{thm:FundamentalTheoremModelWeak} the map sending $p \mapsto \psi_p$ is a bijection with the weak closure of $A$ (if $\varphi_{a_i}$ form a weakly convergent net then, possibly passing to a sub-net, we may assume that $\tp_\varphi(a_i/M)$ converge), which is weakly compact.
    Since its inverse is clearly continuous, it is a homeomorphism.
  \item[\impnum{1}{3}]
    Let $p \in \tS_\varphi(M)$ be defined by $\psi$.
    Since we are only interested in a single formula, we may assume that the language is countable, and find a separable (or countable) $\bM_0 \preceq \bM$ containing the sequence $\{a_n\}$.
    Let $Y = \tS_{\tilde \varphi}(M_0)$, so we have $X \twoheadrightarrow Y$ and $\psi \in C(Y) \subseteq C(X)$ also defines the restriction $p_0 = p\rest_{M_0}$.
    Since $M_0$ is separable, there exists a sub-sequence $a_{n_k}$ such that $\tp_\varphi(a_{n_k}/M_0) \rightarrow p_0$.
    Since $\varphi$ is stable in $M_0$, $\varphi_{a_{n_k}} \rightarrow \psi$ point-wise on $Y$ and therefore on $X$.
  \item[\impfirst]
    By the Eberlein-Šmulian Theorem (see Whitley \cite{Whitley:EberleinSmulianProof}), and since point-wise convergence of a bounded sequence in $C(X)$ implies weak convergence (since $X$ is compact, by the Dominated Convergence Theorem), $A$ is relatively weakly compact, so $\varphi$ is stable in $\bM$.
  \end{cycprf*}
  For the moreover part, just argue as above, taking $\bM_0$ to contain the (countably many) parameters needed for the definition $\psi$, and taking $a_n$ to be any sequence in $M_0$ such that $\tp_\varphi(a_n/M_0) \rightarrow p_0$.
\end{proof}

The second point is with respect to the form of the defining $\tilde \varphi$-predicate, and in particular uniform definability when the formula is stable in the theory.

\begin{fct}[Mazur's Lemma]
  \label{fct:MazurLemma}
  Let $E$ be a Banach space, and let $A \subseteq E$.
  Then the weak closure of $A$ is contained in the closure of the convex hull of $A$.
\end{fct}
\begin{proof}
  Since a closed convex set is weakly closed (Hahn-Banach Theorem, see Brezis \cite{Brezis:AnalyseFonctionnelle}).
\end{proof}

\begin{cor}
  \label{cor:DefiningFormula}
  Let $\varphi(x,y)$ be a formula.
  \begin{enumerate}
  \item If $\varphi$ is stable in a structure $\bM$ then the definition of a type $p \in \tS_\varphi(M)$ can be written as a uniform limit of formulae of the form $\frac{1}{n} \sum_{i<n} \varphi(a_i,y)$, where $a_i \in M$.
  \item If $\varphi$ is $\{0,1\}$-valued, as in classical logic, the definition can be written as a single ``majority rule'' combination of instances $\varphi(a_i,y)$.
  \item If $\varphi$ is stable in a theory $T$, this can be done uniformly for all $\varphi$-types over models (i.e., with the rate of uniform convergence, or number of instances of which a majority is required, fixed).
  \end{enumerate}
\end{cor}
\begin{proof}
  The first item is by Mazur's Lemma, and implies the second.
  For the third item, add a new unary predicate $P$.
  Then it is expressible in first order continuous logic that $P$ is the distance to an elementary sub-structure, and a standard compactness argument yields that if $\varphi$-types realised in models of $T$ over elementary sub-models were not uniformly definable in this fashion, one would not be definable at all, and we are done.
\end{proof}

\providecommand{\bysame}{\leavevmode\hbox to3em{\hrulefill}\thinspace}


\begin{thebibliography}{BBHU08}

\bibitem[BBF11]{BenYaacov-Berenstein-Ferri:ReflexiveRepresentability}
Itaï \bgroup\scshape{}{Ben Yaacov}\egroup{}, Alexander
  \bgroup\scshape{}Berenstein\egroup{}, and Stefano
  \bgroup\scshape{}Ferri\egroup{},
  \href{http://math.univ-lyon1.fr/~begnac/articles/WAPStability.pdf}
  {\emph{Reflexive representability and stable metrics}}, Mathematische
  Zeitschrift \textbf{267} (2011), no.~1-2, 129--138,
  \href{http://dx.doi.org/10.1007/s00209-009-0612-x}{doi:10.1007/s00209-009-0612-x},
  \href{http://arxiv.org/abs/0901.1003}{arXiv:0901.1003}.

\bibitem[BBHU08]{BenYaacov-Berenstein-Henson-Usvyatsov:NewtonMS}
Itaï \bgroup\scshape{}{Ben Yaacov}\egroup{}, Alexander
  \bgroup\scshape{}Berenstein\egroup{}, C.~Ward
  \bgroup\scshape{}Henson\egroup{}, and Alexander
  \bgroup\scshape{}Usvyatsov\egroup{},
  \href{http://math.univ-lyon1.fr/~begnac/articles/mtfms.pdf} {\emph{Model
  theory for metric structures}}, Model theory with Applications to Algebra and
  Analysis, volume 2 (Zoé \bgroup\scshape{}Chatzidakis\egroup{}, Dugald
  \bgroup\scshape{}Macpherson\egroup{}, Anand \bgroup\scshape{}Pillay\egroup{},
  and Alex \bgroup\scshape{}Wilkie\egroup{}, eds.), London Math Society Lecture
  Note Series, vol. 350, Cambridge University Press, 2008, pp.~315--427.

\bibitem[Bre83]{Brezis:AnalyseFonctionnelle}
Ha{\"{\i}}m \bgroup\scshape{}Brezis\egroup{}, \emph{Analyse fonctionnelle},
  Collection Math\'ematiques Appliqu\'ees pour la Ma\^\i trise, Masson, Paris,
  1983, Th{\'e}orie et applications.

\bibitem[BT]{BenYaacov-Tsankov:WAP}
Itaï \bgroup\scshape{}{Ben Yaacov}\egroup{} and Todor
  \bgroup\scshape{}Tsankov\egroup{},
  \href{http://math.univ-lyon1.fr/~begnac/articles/wap-stability.pdf}
  {\emph{Weakly almost periodic functions, model-theoretic stability, and
  minimality of topological groups}}, preprint,
  \href{http://arxiv.org/abs/1312.7757}{arXiv:1312.7757}.

\bibitem[BU10]{BenYaacov-Usvyatsov:CFO}
Itaï \bgroup\scshape{}{Ben Yaacov}\egroup{} and Alexander
  \bgroup\scshape{}Usvyatsov\egroup{},
  \href{http://math.univ-lyon1.fr/~begnac/articles/cfo.pdf} {\emph{Continuous
  first order logic and local stability}}, Transactions of the American
  Mathematical Society \textbf{362} (2010), no.~10, 5213--5259,
  \href{http://dx.doi.org/10.1090/S0002-9947-10-04837-3}{doi:10.1090/S0002-9947-10-04837-3},
  \href{http://arxiv.org/abs/0801.4303}{arXiv:0801.4303}.

\bibitem[Gro52]{Grothendieck:CriteresDeCompacite}
Alexandre \bgroup\scshape{}Grothendieck\egroup{}, \emph{Critères de compacité
  dans les espaces fonctionnels généraux}, American Journal of Mathematics
  \textbf{74} (1952), 168--186.

\bibitem[Iov99]{Iovino:StableBanach}
José \bgroup\scshape{}Iovino\egroup{}, \emph{Stable {B}anach spaces and
  {B}anach space structures, {I} and {II}}, Models, algebras, and proofs
  (Bogotá, 1995), Lecture Notes in Pure and Appl. Math., vol. 203, Dekker, New
  York, 1999, pp.~77--117.

\bibitem[Jam64]{James:UniformlyNonSquareBanachSpaces}
Robert~C. \bgroup\scshape{}James\egroup{}, \emph{Uniformly non-square {B}anach
  spaces}, Annals of Mathematics. Second Series \textbf{80} (1964), 542--550.

\bibitem[KM81]{Krivine-Maurey:EspacesDeBanachStables}
Jean-Louis \bgroup\scshape{}Krivine\egroup{} and Bernard
  \bgroup\scshape{}Maurey\egroup{}, \emph{Espaces de {B}anach stables}, Israel
  Journal of Mathematics \textbf{39} (1981), no.~4, 273--295.

\bibitem[Meg03]{Megrelishvili:FragmentabilityAndRepresentations}
Michael \bgroup\scshape{}Megrelishvili\egroup{}, \emph{Fragmentability and
  representations of flows}, Proceedings of the 17th {S}ummer {C}onference on
  {T}opology and its {A}pplications, vol.~27, 2003, pp.~497--544.

\bibitem[Pil83]{Pillay:IntroductionToStability}
Anand \bgroup\scshape{}Pillay\egroup{}, \emph{An introduction to stability
  theory}, Oxford Logic Guides, vol.~8, The Clarendon Press Oxford University
  Press, New York, 1983.

\bibitem[She90]{Shelah:ClassificationTheory}
Saharon \bgroup\scshape{}Shelah\egroup{}, \emph{Classification theory and the
  number of nonisomorphic models}, second ed., Studies in Logic and the
  Foundations of Mathematics, vol.~92, North-Holland Publishing Co., Amsterdam,
  1990.

\bibitem[Whi67]{Whitley:EberleinSmulianProof}
Robert \bgroup\scshape{}Whitley\egroup{}, \emph{An elementary proof of the
  {E}berlein-\v {S}mulian theorem}, Mathematische Annalen \textbf{172} (1967),
  116--118.

\end{thebibliography}
\end{document}